\documentclass[11pt, reqno, a4paper]{amsart}
\usepackage[margin=1.2in]{geometry}
\numberwithin{equation}{section}
\usepackage{amssymb,amsfonts,amsthm}
\usepackage[utf8]{inputenc}
\usepackage{listings}
\usepackage{bm}

\usepackage{color}
\usepackage[bookmarks]{hyperref}

\allowdisplaybreaks[4]

\newtheoremstyle{myremark}{10pt}{10pt}{}{}{\bfseries}{.}{.5em}{}

 \newtheorem{thm}{Theorem}[section]
 \newtheorem{cor}[thm]{Corollary}
 \newtheorem{lem}[thm]{Lemma}
 \newtheorem{prop}[thm]{Proposition}
 \theoremstyle{definition}

 \newtheorem{rem}[thm]{Remark}

\allowdisplaybreaks[4]

\begin{document}

\title[$L^p$ Hardy inequalities with homogeneous weights]{$L^p$ Hardy inequalities with homogeneous weights}

\author[ S. Roy]{ Subhajit Roy$^{1,2}$}

\keywords{Hardy inequalities, Laplace-Beltrami operators, Gagliardo-Nirenberg inequalities.}

\subjclass{ 35A23, 46E35, 49J40.}

{\email{ rsubhajit.math@gmail.com, subhajit.roy@iiserpune.ac.in}}

\maketitle
\centerline{$^{1}$Department of Mathematics, Indian Institute of Technology Madras,
 }
\centerline{Chennai  600036, India}
\centerline{$^{2}$Department of Mathematics, Indian Institute of Science Education and Research, Pune
 }
\centerline{Pune 411008, India}
\begin{abstract} 
For $p\in (1,\infty)$ and $\alpha\in\mathbb{R}$, we consider measurable functions $g$ on $\mathbb{S}^{N-1}$ that satisfy the following weighted Hardy inequality:
  $$  \int_{\mathbb{R}^N}\frac{ g (x/|x|)}{|x|^{p+\alpha}}|u(x)|^p dx \leq C\int_{\mathbb{R}^N}\frac{|\nabla u(x)|^p}{|x|^\alpha} dx, \quad\forall\,u\in \mathcal{C}_c^\infty(\mathbb{R}^N),$$
for some constant $C>0$. Depending on $N$, $p$, and $\alpha$, we identify suitable function spaces for $g$ so that the above inequality holds. The constant obtained is sharp, in the sense that it is sharp when $g \equiv 1$. Furthermore, we establish the sharp fractional Hardy inequality with homogeneous weights.

\end{abstract}

\section{Introduction}
For $p\in(1,N)$, the classical {\it{Hardy inequality}} states that 
\begin{equation*}
    \int_{\mathbb{R}^N}\frac{|u(x)|^p}{|x|^p}dx\leq \left(\frac{p}{N-p}\right)^p\int_{\mathbb{R}^N}|\nabla u(x)|^pdx,\quad\forall\,u\in \mathcal{C}_c^\infty(\mathbb{R}^N).
\end{equation*} The constant \(\left(\frac{p}{N-p}\right)^p\) is sharp but never attained (see \cite[Lemma 2.1]{peral1998}). Hardy's inequality plays a fundamental role in mathematical physics, harmonic analysis, spectral theory, the analysis of linear and nonlinear PDEs, and stochastic analysis. Over the years, this inequality has been generalised in several directions, most notably by replacing the weight \(\frac{1}{|x|^p}\) with more general Hardy weights. The admissible function spaces for such weights are discussed in \cite{Allegretto1992, Anoop2021, bessel, Visciglia2005} and the references therein. Characterisations of Hardy weights are given in \cite{UjjalDas2023, pinc, maazya1985}, while their fractional and discrete counterparts are studied in \cite{dyda-ch} and \cite{Pinchover-dis}, respectively. 
Hardy-type inequalities with respect to the Gaussian measure are address in \cite{Barbara2007}.
A comprehensive review of Hardy inequalities is presented in \cite{Hardybook2015}.  For further readings on Hardy inequalities, we refer to the monograph \cite{mazya2011}. Another extension of Hardy’s inequality is the \textit{Caffarelli-Kohn-Nirenberg inequality}, established by Caffarelli, Kohn, and Nirenberg in \cite{caffarelli1984}. We state a particular case of their result. Let $N\ge 1,\,p> 1,$ and $\alpha\in \mathbb{R}.$ If $N>p+\alpha,$ then the following weighted Hardy inequality holds:
\begin{equation}\label{hardy-phi}
    \int_{\mathbb{R}^N}\frac{|u(x)|^p}{|x|^{p+\alpha}} dx \leq C\int_{\mathbb{R}^N}\frac{|\nabla u(x)|^p}{|x|^\alpha} dx, \quad\forall\,u\in \mathcal{C}_c^\infty(\mathbb{R}^N),
\end{equation}
 for some constant $C>0.$ The optimal constant in \eqref{hardy-phi} is $C=\left(\frac{p}{N-p-\alpha}\right)^p$ (see \cite[Theorem 1.1]{catrina2001} for $p=2$ and \cite[Theorem 2.1]{cylin} for any $p>1$).
\smallskip

Hardy inequalities have also been studied in domains with boundary singularities. For instance, Davies \cite[Page 425]{davies1995} considered a conical region in $\mathbb{R}^N$ of the form 
$$\Omega=\{(r,\vartheta): 0<r<\delta,\,\vartheta\in U\},$$ where $\delta>0 $ and $U$ is an open set in $\mathbb{S}^{N-1}$ with smooth boundary. For such domains, he established a Hardy inequality involving the distance to the conical part of the boundary, defined as $\partial_1 \Omega:=\{(r,\vartheta): 0\le r<\delta, \vartheta\in \partial U\}.$ More precisely, there exists a constant $C>0,$ such that
\begin{equation*}
    \int_\Omega \frac{|u(x)|^2}{d^2(x)}\le C
    \int_\Omega|\nabla u(x)|^2 dx, \quad\forall\,u\in \mathcal{C}_c^\infty(\Omega),
\end{equation*}
where $d(x)$ denotes the distance from $x
\in \Omega$ to $\partial_1 \Omega$.
A key geometric observation is that for any point $x=(r,\vartheta)\in \Omega,$ the distance admits the factorisation $d(x)=rT(\vartheta)$ (see \cite[Page 423]{davies1995}), where   $T$ is a continuous function on $U$ satisfying $0<T(\vartheta)\le 1$ for all $\vartheta\in U$ and $\lim_{\vartheta\to \partial U}T(\vartheta)=0.$ Setting  $g=1/T^2$, the above inequality can be rewritten as
 \begin{equation*}
    \int_\Omega \frac{g(x/|x|)}{|x|^2}|u(x)|^2\le C
    \int_\Omega|\nabla u(x)|^2 dx, \quad\forall\,u\in \mathcal{C}_c^\infty(\Omega).
\end{equation*} Thus, the angular weight $g$ arises naturally from the geometry of the boundary.

\smallskip

Motivated by these developments, in this article we consider a further generalisation. 
 Instead of the radial weight $1/|x|^{p+\alpha}$ on the left-hand side of \eqref{hardy-phi}, we consider a more general class of homogeneous functions of degree $-(p+\alpha)$, which may also have singularities along the rays starting at the origin. 
More precisely, we look for a class of measurable functions $g$ on $\mathbb{S}^{N-1}$ so that the following inequality holds
\begin{equation}\label{hom-we}
    \int_{\mathbb{R}^N}\frac{ g (x/|x|)}{|x|^{p+\alpha}}|u(x)|^p dx \leq C\int_{\mathbb{R}^N}\frac{|\nabla u(x)|^p}{|x|^\alpha} dx, \quad\forall\,u\in \mathcal{C}_c^\infty(\mathbb{R}^N),
\end{equation}
for some constant $C>0.$ 

\smallskip

There is significant interest in studying the above inequality in certain cases. For $p=2$ with $\alpha=0$, Hoffmann-Ostenhof and Laptev \cite{ari2015} studied \eqref{hom-we} with the sharp constant, using a sharp estimate for the first negative eigenvalue of the Schr\"odinger operator in $L^2({\mathbb{S}^{N-1}})$.

\smallskip

For a measurable function $g$ defined on $\mathbb{S}^{N-1}$, we denote by $\|\cdot\|_{L^q(\mathbb{S}^{N-1})}$ the quantity
$$\| g \|_{L^q(\mathbb{S}^{N-1})}=\left(\int_{\mathbb{S}^{N-1}}| g (\vartheta)|^qd\vartheta\right)^{1/q},\quad q\in [1,\infty),$$ where $d\vartheta$ is the measure induced by the Lebesgue measure on $\mathbb{S}^{N-1} \subset \mathbb{R}^N$. Hoffmann-Ostenhof and Laptev proved the following theorem:
\begin{thm}\cite[Theorem 1.1]{ari2015}\label{sharp-thm}
    Let $N\ge 3$ and $0\le g\in L^q(\mathbb{S}^{N-1}),$ where $q={\frac{(N-2)^2}{2(N-1)}+1}.$ Then the following inequality holds
    \begin{equation}\label{sharp-Hardy}
    \int_{\mathbb{R}^N}\frac{ g (x/|x|)}{|x|^{2}}|u(x)|^2 dx \leq C \int_{\mathbb{R}^N}|\nabla u(x)|^2 dx, \quad\forall\,u\in \mathcal{C}_c^\infty(\mathbb{R}^N),    
    \end{equation}
     where 
    $$C = \frac{4\,\|g \|_{L^q(\mathbb{S}^{N-1})}}{(N-2)^2\,|\mathbb{S}^{N-1}|^{1/q} }.$$
Moreover, the inequality \eqref{sharp-Hardy} is sharp and reduces to the classical sharp form when $g\equiv 1.$
\end{thm}
Inequality \ref{sharp-Hardy} was also studied in \cite{Laptev2020, Laptevlatest} under suitable sufficient conditions on the weight function $g$. 
We are interested in studying a weighted extension of \eqref{sharp-Hardy} in a more general setting.
Depending on the values of $N$ and $p$, we identify a suitable function space for $g$ such that \eqref{hom-we} holds. 
We now state our first result:
\begin{thm}\label{thm-p2}
Let $N\ge 2,\,p\in (1,\infty),$ and $\alpha\in \mathbb{R}$ be such that $N> p+\alpha$. 
Assume that $g \in L^q(\mathbb{S}^{N-1})$, where
\[
q =
\begin{cases}
\dfrac{N-1}{p}, & \text{if } p < N-1,\\
1, & \text{if } p > N-1 ,
\end{cases}
\]
and  $q>1$ when $ p = N-1.$
Then there exists a constant $C>0$, independent of $u,$ such that
\begin{equation}\label{eq-p02}
    \int_{\mathbb{R}^N}\frac{ g (x/|x|)}{|x|^{p+\alpha}}|u(x)|^p dx \leq C\| g \|_{L^{q}(\mathbb{S}^{N-1})} \int_{\mathbb{R}^N}\frac{|\nabla u(x)|^p}{|x|^\alpha} dx, \quad\forall\,u\in \mathcal{C}_c^\infty(\mathbb{R}^N).
\end{equation}
\end{thm}

\smallskip

The proof of Theorem \ref{thm-p2} is based on Sobolev inequalities on the unit sphere (see \cite{Aubin1976, Aubin2}, Theorem \ref{thm-sphere}), together with one-dimensional weighted Hardy inequalities.
\begin{rem}
    Notice that \eqref{hardy-phi} follows from \eqref{eq-p02} by talking $g\equiv 1\in L^q(\mathbb{S}^{N-1}).$ Moreover, Theorem \ref{thm-p2} is stronger than Theorem \ref{sharp-thm} in the sense that it admits a larger function space $L^q(\mathbb{S}^{N-1})$ than $L^r(\mathbb{S}^{N-1}),$ where $r={\frac{(N-2)^2}{2(N-1)}+1}$.
\end{rem}
\begin{rem}
  Let $p\in (1,N)$. In \cite{Anoop2021}, it was shown that for $h\in L^{\frac{N}{p},\infty}(\mathbb{R}^N)$,  
\begin{equation}\label{weighted-hardy-in}
    \int_{\mathbb{R}^N}|h(x)||u(x)|^pdx\leq C\|h\|_{L^{\frac{N}{p},\infty}(\mathbb{R}^N)} \int_{\mathbb{R}^N}|\nabla u(x)|^pdx,\quad\forall\,u\in \mathcal{C}_c^1(\mathbb{R}^N),
\end{equation}
where  $C=C(N,p)>0$ and $L^{\frac{N}{p},\infty}$ denotes the Lorentz space with quasi-norm $  \|u\|_{q,\infty}:=\sup_{t>0}|\{x\in \mathbb{R}^N:|u(x)|>t\}|^\frac{1}{q}t.$
For $h(x)=\frac{ g (x/|x|)}{|x|^p}$ with $ g \in L^\frac{N}{p}(\mathbb{S}^{N-1})$,
\eqref{weighted-hardy-in}  becomes 
\begin{equation*}
    \int_{\mathbb{R}^N}\frac{ g (x/|x|)}{|x|^p}|u(x)|^pdx\leq C_1\int_{\mathbb{R}^N}|\nabla u(x)|^pdx,\quad\forall\,u\in \mathcal{C}_c^1(\mathbb{R}^N),
\end{equation*}
where $C_1=C(1/N)^\frac{p}{N} \| g\|_{L^\frac{N}{p}(\mathbb{S}^{N-1})}.$ While this inequality holds for all $g\in L^\frac{N}{p}(\mathbb{S}^{N-1}),$  Theorem \ref{thm-p2} extends it to the larger class $L^q(\mathbb{S}^{N-1})$, thereby allowing a broader range of admissible weights on the unit sphere.
\end{rem}
Now, we study the inequality \eqref{eq-p02} with the \textit{sharp constant} for $p=2$. The constant is sharp in the sense that it is sharp when the weight function $g$ is constant.

\begin{thm}\label{sharp-wthm}  Let $N\ge 2$ and $\alpha\in \mathbb{R}$ be such that $N>2+\alpha$, and $g\in L^q(\mathbb{S}^{N-1})$ with $q>1$. 
\begin{enumerate}
    \item If $2N\alpha<(1+\alpha)^2$ and $q=\frac{(N-\alpha-2)^2}{2(N-1)}+1,$ then 
    \begin{equation}\label{shrp-whardy}
         \int_{\mathbb{R}^{N}} \frac{g(x/|x|)}{|x|^{2+\alpha}} |u|^2dx\le C\int_{\mathbb{R}^N}\frac{|\nabla u|^2}{|x|^\alpha}dx, \quad\forall\,u\in \mathcal{C}_c^\infty(\mathbb{R}^N),
    \end{equation}
    where 
    \begin{equation}\label{constant-best}
        C= \frac{4 \,\|g \|_{L^q(\mathbb{S}^{N-1})}}{(N-\alpha-2)^2\,|\mathbb{S}^{N-1}|^{1/q} }.
    \end{equation}
\item If $2N\alpha<(1+\alpha)^2,$  $N>3$, and $q=\frac{N-1}{2}$, then 
\begin{equation}\label{sharp-whardy2}
    \int_{\mathbb{R}^{N}} \frac{g(x/|x|)}{|x|^{2+\alpha}} |u|^2dx\le   \gamma_0 C\int_{\mathbb{R}^N}\frac{|\nabla u|^2}{|x|^\alpha}dx-(\gamma_0-1)\frac{\|g \|_{L^q(\mathbb{S}^{N-1})}}{|\mathbb{S}^{N-1}|^{1/q} }\int_{\mathbb{R}^N}\frac{|u|^2}{|x|^{2+\alpha}}dx, \;\;\forall\,u\in \mathcal{C}_c^\infty(\mathbb{R}^N),
\end{equation} 
   where $C$ is as in \eqref{constant-best} and $$\gamma_0=\frac{(N-\alpha-2)^2}{(N-1)(N-3)}>1.$$
  \item  If $2N\alpha\ge (1+\alpha)^2,$ $N>3,$  and $q=\frac{N-1}{2},$ then 
    $$ \int_{\mathbb{R}^{N}} \frac{g(x/|x|)}{|x|^{2+\alpha}} |u|^2dx\le C \int_{\mathbb{R}^N}\frac{|\nabla u|^2}{|x|^\alpha}dx, \quad\forall\,u\in \mathcal{C}_c^\infty(\mathbb{R}^N),$$
     where $C$ is as in \eqref{constant-best}.
    \end{enumerate}
     Moreover, the above three inequalities are sharp as it is sharp when $ g\equiv 1.$
\end{thm}
\begin{rem} $(i)$
 Under the condition $2N\alpha < (1+\alpha)^2$ and $\alpha>0$ with $N > 3$, inequality \eqref{sharp-whardy2} holds for a broader class of functions $g$ than inequality \eqref{shrp-whardy}, since $$L^{\frac{(N-\alpha-2)^2}{2(N-1)}+1}(\mathbb{S}^{N-1})\subsetneq L^\frac{N-1}{2}(\mathbb{S}^{N-1}).$$ Moreover, the assumption $2N\alpha < (1+\alpha)^2$ guarantees that $\gamma_0 > 1.$ 

\smallskip

$(ii)$
 The sharp inequality \eqref{sharp-Hardy} is obtained from \eqref{shrp-whardy} by setting $\alpha = 0$. Thus, Theorem \ref{sharp-wthm} extends Theorem \ref{sharp-thm} to the weighted case. The proof of Theorem \ref{sharp-wthm} relies on the \textit{Gagliardo-Nirenberg inequalities} on the unit sphere, established by Dolbeault, Esteban, and Laptev \cite{Ari2014} (see Lemma \ref{G-N-ineq}).

\end{rem}
\begin{rem}
In Theorem \ref{theorem2.1}, we also study \eqref{hom-we} for weight functions $g$ in an $L^q$ space, which may differ from the function spaces considered for $g$ in Theorem \ref{thm-p2} and Theorem \ref{sharp-wthm}. Moreover, Theorem \ref{theorem2.1} recovers the sharp weighted Hardy inequality in the special case $g \equiv 1$.
\end{rem}
Next, we obtain a weighted fractional Hardy inequality along with its sharp constant. In \cite{Frank2008}, Frank and Seiringer established the fractional Hardy inequality with the sharp constant. Let $N \geq 1$, $p \in [1,\infty)$, and $s \in (0,1)$ be such that $N > sp$. Then, Theorem 1.1 in \cite{Frank2008} establishes the following \textit{fractional Hardy inequality}:
 \begin{equation}\label{frac-norm}
   \int_{\mathbb{R}^N}\frac{|u(x)|^p}{|x|^{sp}} dx\leq \Lambda_{N,s,p}  \int_{\mathbb{R}^N}\int_{\mathbb{R}^N}\frac{|u(x)-u(y)|^p}{|x-y|^{N+sp}}dxdy,\quad\forall\,u\in \mathcal{C}^\infty_c(\mathbb{R}^N)
\end{equation}
  with 
    \begin{equation*}
       \Lambda_{N,s,p}^{-1}:=2\int_0^1r^{sp-1}|1-r^{(N-sp)/p}|^p\Psi_{N,s,p}(r) \,dr,
 \end{equation*}
 \begin{equation*}
     \Psi_{N,s,p}(r):=|\mathbb{S}^{N-2}|\int_{-1}^1\frac{(1-t^2)^{(N-3)/2}}{(1-2rt+r^2)^{(N+sp)/2}}\,dt,\quad N\geq2,
 \end{equation*}
 \begin{equation*}
     \Psi_{1,s,p}(r):=\left(\frac{1}{(1-r)^{1+sp}}+\frac{1}{(1+r)^{1+sp}}\right),\quad N=1.
 \end{equation*}
The constant $\Lambda_{N,s,p}$ is sharp. These inequalities are of great interest in the last decades, given their significant applications in stochastic processes and partial differential equations (see, for example, \cite{Bogdan22003,  dyda2011} and the references therein). 
For the generalisation of the Hardy inequality to the Orlicz setting, see \cite{cianchi2021, subha1, subha2, Cianchi1999,  Subha3} and the references therein.

\smallskip

For $p = 2$, Hoffmann-Ostenhof and Laptev \cite{ari2015} studied a weighted extension of \eqref{frac-norm} with the sharp constant. More precisely, Theorem 1.7 in \cite{ari2015} states that for any $g \in L^\frac{N}{2s}(\mathbb{S}^{N-1})$ with $N > 2s$, the following weighted fractional Hardy inequality holds:
\begin{equation}\label{frac-p=2}
        \int_{\mathbb{R}^N}\frac{ g (x/|x|)}{|x|^{2s}}|u(x)|^2 dx\leq \Lambda_{N,s,2}\frac{\|g \|_{L^\frac{N}{2s}(\mathbb{S}^{N-1})}}{|\mathbb{S}^{N-1}|^{2s/N}}\int_{\mathbb{R}^N}\int_{\mathbb{R}^N}\frac{|u(x)-u(y)|^2}{|x-y|^{N+2s}}dxdy,
\end{equation} 
for every $u\in \mathcal{C}^\infty_c(\mathbb{R}^N)$. Note that if $g$ is a constant function, then the above inequality reduces to the sharp fractional Hardy inequality \eqref{frac-norm} for $p = 2$. Now, we have the following sharp extension of the above inequality.
\begin{thm}\label{frac-hardy}
    Let $N\ge 1,\,p\in [1,\infty),$ and $s\in (0,1)$ be such that $N>sp$. If $ 0\le g \in L^\frac{N}{sp}(\mathbb{S}^{N-1})$, then the following inequality holds:
    \begin{equation}\label{eqn-frac-p}
    \int_{\mathbb{R}^N}\frac{ g (x/|x|)}{|x|^{sp}}|u(x)|^p dx\leq  C \int_{\mathbb{R}^N}\int_{\mathbb{R}^N}\frac{|u(x)-u(y)|^p}{|x-y|^{N+sp}}dxdy,\quad \, \forall\,u\in \mathcal{C}^\infty_c(\mathbb{R}^N).
    \end{equation}
 where the constant  \begin{equation*}\label{tau}
        C=\Lambda_{N,s,p}\frac{\| g \|_{L^\frac{N}{sp}(\mathbb{S}^{N-1})} }{|\mathbb{S}^{N-1}|^{sp/N}}.
     \end{equation*}
 is sharp in the sense that if $g\equiv 1$, then \eqref{eqn-frac-p} takes the sharp form of the fractional Hardy inequality.
\end{thm}
\begin{rem}
Notice that the inequalities \eqref{frac-norm} and \eqref{frac-p=2} follow from \eqref{eqn-frac-p} by taking $g\equiv 1$ and $p=2$, respectively. 
\end{rem}

\smallskip

The rest of the article is organised as follows: Section \ref{section-peli} presents some known results that are essential for this article. The proofs of Theorem \ref{thm-p2}, Theorem \ref{sharp-wthm}, Theorem \ref{frac-hardy}, and Theorem \ref{theorem2.1} are given in Section \ref{section-main}. 

\section{Preliminaries}\label{section-peli} In this section, we first describe symmetrisation and some of its properties. We then state several known results and inequalities that are used to prove the theorems presented in this article.

\subsection{Symmetrisation}
Let $\Omega\subset \mathbb{R}^N$ be a measurable set and $f$ be a measurable function on $ \mathbb{R}^N$. Now, we define the following notions:
\begin{itemize}
\item \textbf{Symmetric rearrangement $\Omega^*$:}  The
symmetric rearrangement of the set $\Omega$, to be the open ball centred at the origin
whose volume is that of $\Omega$. Thus
\begin{equation}\label{symm-set}
   \Omega^*=\{x\,:\,|x|<r\},\quad\,\text{with}\,\,(|\mathbb{S}^{N-1}|/N)r^N=|\Omega|, 
\end{equation}
 where $|E|$ denotes the Lebesgue measure of a set $E\subset\mathbb{R}^N$.
\item \textbf{Symmetric decreasing rearrangement $f^*$:}  The symmetric decreasing rearrangement, $f^*$ is defined as 
\begin{equation*}
  f^{*}(x)=\int_0^\infty\chi_{\{y\in  \mathbb{R}^N:\,|f(y)|>t\}^*}(x)\,dt 
  \end{equation*}
  where $\chi_A$ denotes the characteristic function of $A\subset\mathbb{R}^N$.
\end{itemize}

Next, we state three important inequalities related to symmetrisation. 
\begin{prop}\label{prop1}
Let $N\ge 1$. Then the following inequalities hold:
\begin{enumerate}
    \item \textbf{Hardy-Littlewood inequality \cite[Theorem 3.2.19]{Edmunds2004}:}  Let $u$ and $v$ be two nonnegative measurable functions. Then $$\int_{\mathbb{R}^N} u(x)v(x) dx\leq \int_{\mathbb{R}^N}u^*(x)v^*(x)dx.$$
    \item  \textbf{Weighted P\'olya-Szeg\"o inequality \cite[Theorem 8.1]{alvino}:}  Let  $p\in [1,\infty)$ and $\alpha\in \mathbb{R}$ be such that $N>p+\alpha.$ Then for every  $u\in \mathcal{C}_c^\infty(\mathbb{R}^N)$, $$\int_{\mathbb{R}^N}\frac{|\nabla u^*(x)|}{|x|^\alpha}dx\leq \int_{\mathbb{R}^N}\frac{|\nabla u(x)|}{|x|^\alpha}dx.$$ 
    \item  \textbf{Fractional P\'olya-Szeg\"o inequality \cite[Theorem 9.2]{Lieb1989}:}  Let $p\in [1,\infty)$. Then for every  $u\in \mathcal{C}_c^\infty(\mathbb{R}^N)$, $$\int_{\mathbb{R}^N}\int_{\mathbb{R}^N}\frac{|u(x)-u(y)|^p}{|x-y|^{N+sp}}dxdy\leq \int_{\mathbb{R}^N}\int_{\mathbb{R}^N}\frac{|u^*(x)-u^*(y)|^p}{|x-y|^{N+sp}}dxdy.$$ 
\end{enumerate}
\end{prop}

\subsection{Sobolev embedding}
For $p\in (1,\infty),$ let $H^p_1(\mathbb{S}^{N-1})$ denote the Sobolev space of order $p,$ that is the completion of $C^\infty(\mathbb{S}^{N-1})$ with respect to the norm $$\|u\|^p_{H_1^p(\mathbb{S}^{N-1})}=\|u\|_{L^p(\mathbb{S}^{N-1})}^p+\|\nabla_\vartheta u\|_{L^p(\mathbb{S}^{N-1})}^p,$$ where $\Delta_\vartheta$ is the Laplace-Beltrami operator on the unit sphere $\mathbb{S}^{N-1}\subset \mathbb{R}^N$. We now state the Sobolev inequalities for $H^p_1(\mathbb{S}^{N-1})$ (see \cite[Theorem 2.29, Theorem 2.51, Lemma 2.2]{Aubin2}, \cite{Aubin1976, bestconst}).
\begin{thm}\label{thm-sphere}
Let $N \geq 2$, $p \in (1,\infty),$ and $r_0>1.$ Then there exists a constant 
$C > 0$, independent of $u$, such that  
\[
\|u\|_{L^r(\mathbb{S}^{N-1})}^p 
\;\leq\; 
C \left( \|\nabla_\vartheta u\|_{L^p(\mathbb{S}^{N-1})}^p 
       + \|u\|_{L^p(\mathbb{S}^{N-1})}^p \right),\quad\forall \,u \in H_1^p(\mathbb{S}^{N-1}),
\]
where 
\[
r = 
\begin{cases}
\dfrac{(N-1)p}{N-p-1}, & \text{if } p < N-1, \\
r_0, & \text{if } p = N-1, \\
\infty, & \text{if } p > N-1. 
\end{cases}
\]
\end{thm}

\subsection{Weighted Hardy inequality:}  We recall the one-dimensional weighted Hardy inequality obtained by Opic and Kufner in \cite[Lemma 1.3]{Kufner2}, which will be used in the proofs of Theorem \ref{thm-p2} and Theorem \ref{sharp-wthm}.
 Let $AC_R(0,\infty)$ denote the set of all functions that are absolutely continuous on every compact subinterval $[a,b] \subset (0,\infty)$ and satisfy $\lim_{t \to \infty} u(t) = 0.$

\begin{lem}\label{weighted-Hardy}
Let $p\in (1,\infty)$ and $\beta>p-1.$ Then the following weighted Hardy inequality holds:
    \begin{equation*}
  \int_0^{\infty}|f(r)|^pr^{\beta-p}dr\le \left(\frac{p}{\beta-p+1}\right)^p  \int_0^{\infty}|f^\prime(r)|^pr^\beta dr,\quad\forall\,u\in AC_R(0,\infty).
\end{equation*}
The constant in this inequality is sharp.
\end{lem}

\section{Weighted Hardy inequality}\label{section-main}
In this section, we prove all the theorems introduced in the Introduction. We start with a lemma that establishes a weighted Hardy-type inequality on the unit sphere $\mathbb{S}^{N-1}$, which will play a key role in the proof of Theorem \ref{thm-p2}. 
 
\begin{lem}\label{prop-1}
Let $N\ge 2$, $p\in (1,\infty),$ and $\Delta_\vartheta$ be the Laplace-Beltrami operator on $\mathbb{S}^{N-1}$. Assume that  $g\in L^q(\mathbb{S}^{N-1})$, where \[
q =
\begin{cases}
\dfrac{N-1}{p}, & \text{if } p < N-1, \\
1, & \text{if } p > N-1 ,
\end{cases}
\]
and  $q>1$ when $ p = N-1.$
Then there exists a constant $C>0,$ independent of $u,$ such that for every $u\,\in H^p_1(\mathbb{S}^{N-1}),$
\begin{equation*}
    \int_{\mathbb{S}^{N-1}}  g (\vartheta)|u(\vartheta)|^pd\vartheta
  \leq C \| g \|_{L^q(\mathbb{S}^{N-1})} \left(\int_{\mathbb{S}^{N-1}}|u(\vartheta)|^p d\vartheta+\int_{\mathbb{S}^{N-1}}|\nabla_{\vartheta} u(\vartheta)|^pd\vartheta \right).
\end{equation*}
\end{lem}
\begin{proof} Let
$u\in H^p_1(\mathbb{S}^{N-1})$ and $ g \in L^q(\mathbb{S}^{N-1})$. If $p<N-1,$ then H\"older's inequality for the conjugate pair $\left(\frac{N-1}{p},\frac{N-1}{N-p-1}\right)$  yields 
$$ \int_{\mathbb{S}^{N-1}}  g (\vartheta)|u(\vartheta)|^pd\vartheta \leq \| g \|_{L^\frac{N-1}{p}(\mathbb{S}^{N-1})} \|u\|^p_{L^\frac{(N-1)p}{N-p-1}(\mathbb{S}^{N-1})}= \| g \|_{L^q(\mathbb{S}^{N-1})} \|u\|^p_{L^\frac{(N-1)p}{N-p-1}(\mathbb{S}^{N-1})}.$$
Again, applying H\"older's inequality for the conjugate pair $\left(q,q^\prime=\frac{q}{q-1}\right)$, we obtain
 $$\int_{\mathbb{S}^{N-1}}  g (\vartheta)|u(\vartheta)|^pd\vartheta \leq 
\|g\|_{L^q(\mathbb{S}^{N-1})}\|u\|^p_{L^{pq^\prime}(\mathbb{S}^{N-1})},\quad q>1.$$ Moreover, 
$$\int_{\mathbb{S}^{N-1}}  g (\vartheta)|u(\vartheta)|^pd\vartheta \leq \|g\|_{L^1(\mathbb{S}^{N-1})}\|u\|^p_{L^{\infty}(\mathbb{S}^{N-1})}.$$
Thus, collecting the above inequalities, we can write
$$\int_{\mathbb{S}^{N-1}}  g (\vartheta)|u(\vartheta)|^pd\vartheta \leq \|g\|_{L^q(\mathbb{S}^{N-1})}\|u\|^p_{L^r(\mathbb{S}^{N-1})},$$
where  \[
r = 
\begin{cases}
\dfrac{(N-1)p}{N-p-1}, & \text{if } p < N-1, \\
pq^\prime>1, & \text{if } p = N-1, \\
\infty, & \text{if } p > N-1. 
\end{cases}
\]
Hence, by applying Theorem \ref{thm-sphere}, we obtain
\[
\int_{\mathbb{S}^{N-1}}  g (\vartheta)|u(\vartheta)|^pd\vartheta \leq C\|g\|_{L^q(\mathbb{S}^{N-1})}\left( \|u\|^p_{L^p(\mathbb{S}^{N-1})}+\|\nabla_\vartheta u\|^p_{L^p(\mathbb{S}^{N-1})}\right),
\]
where $C>0$ is a constant independent of $u$.

This completes the proof.
\end{proof}

\smallskip

Now, we proceed to prove Theorem \ref{thm-p2} by using Lemma \ref{prop-1}.

\smallskip

\noindent \textbf{Proof of Theorem \ref{thm-p2}:} 
   Let $N\ge 2,\,p\in (1,\infty),$ and $\alpha\in \mathbb{R}$ be such that $N> p+\alpha$.  Let $u\,\in \mathcal{C}_c^\infty(\mathbb{R}^N)$ and  $x=(r,\vartheta)\in \mathbb{R}^N$ denote the polar coordinates in $\mathbb{R}^N.$ Assume that $g\in L^q(\mathbb{S}^{N-1})$, where \[q =
\begin{cases}
\dfrac{N-1}{p}, & \text{if } p < N-1, \\
1, & \text{if } p > N-1 ,
\end{cases}
\] 
and  $q>1$ when $ p = N-1.$ We have
   $$\int_{\mathbb{R}^N} \frac{ g (x/|x|)}{|x|^{p+\alpha}}|u(x)|^pdx= \int_0^\infty \int_{\mathbb{S}^{N-1}}  g (\vartheta)|u|^p  r^{N-p-\alpha-1} d\vartheta dr.$$
Moreover, applying Lemma \ref{prop-1} to $u(r,\vartheta)$ with fixed $r$, we obtain
$$\int_{\mathbb{S}^{N-1}} g (\vartheta)|u|^p d\vartheta\le C \| g \|_{L^q(\mathbb{S}^{N-1})} \left(\int_{\mathbb{S}^{N-1}}|u|^p d\vartheta+\int_{\mathbb{S}^{N-1}}|\nabla_{\vartheta} u|^pd\vartheta \right).$$
Consequently,
\begin{equation}\label{2n-ine}
\int_{\mathbb{R}^N} \frac{ g (x/|x|)}{|x|^{p+\alpha}}|u(x)|^pdx\le C \| g \|_{L^q(\mathbb{S}^{N-1})} \int_0^\infty   \int_{\mathbb{S}^{N-1}}\left(|u|^p+|\nabla_{\vartheta} u|^p\right)r^{N-p-\alpha-1} d\vartheta dr.   
\end{equation}
Next, we estimate the first integral on the right-hand side of \eqref{2n-ine}. Applying Lemma \ref{weighted-Hardy} with $\beta=N-\alpha-1,$ we obtain
\begin{equation*}
     \int_0^{\infty}|f(r)|^pr^{N-p-\alpha-1}d dr\le \left(\frac{p}{N-p-\alpha}\right)^p \int_0^{\infty}|f^\prime(r)|^pr^{N-\alpha-1} dr.
\end{equation*}
Now, applying the above inequality to $u(r,\vartheta)$ for a fixed $\vartheta$ and then integrating over $\mathbb{S}^{N-1},$ we obtain 
\begin{equation*}
     \int_0^{\infty}\int_{\mathbb{S}^{N-1}}|u|^pr^{N-p-\alpha-1}dr\le \left(\frac{p}{N-p-\alpha}\right)^p \int_0^{\infty}\int_{\mathbb{S}^{N-1}}|\partial_r u|^pr^{N-\alpha-1}dr.
\end{equation*}
Set $C_1=\left(\frac{p}{N-p-\alpha}\right)^p.$ Then, from \eqref{2n-ine} we get
 \begin{align}
\int_{\mathbb{R}^N} \frac{ g (x/|x|)}{|x|^{p+\alpha}}|u(x)|^pdx& \le C \| g \|_{L^q(\mathbb{S}^{N-1})}  \int_0^\infty \int_{\mathbb{S}^{N-1}}\left(C_1|\partial_ru|^p+\frac{|\nabla_{\vartheta} u|^p}{r^p}\right) r^{N-\alpha-1}  d\vartheta dr \nonumber\\& \le CC_2 \| g \|_{L^q(\mathbb{S}^{N-1})}  \int_0^\infty \int_{\mathbb{S}^{N-1}}\left(|\partial_ru|^p+\frac{|\nabla_{\vartheta} u|^p}{r^p}\right) r^{N-\alpha-1} d\vartheta dr .  \label{eqn-t1}
    \end{align}
where $C_2=\max\{1,C_1\}.$ Observe that
$$|\partial_ru|^p+\frac{|\nabla_{\vartheta} u|^p}{r^p}\le C_3\left(|\partial_ru|^2+\frac{1}{r^2}|\nabla_{\vartheta} u|^2\right)^\frac{p}{2}=C_3|\nabla u|^p,$$ where $C_3=\max\{1,2^{1-p/2}\}.$
Therefore,
\begin{align*}
     \int_0^{\infty}\int_{\mathbb{S}^{N-1}}\left(|\partial_ru|^p+\frac{|\nabla_{\vartheta} u|^p}{r^p}\right)r^{N-\alpha-1}d\vartheta dr &\le C_3 \int_0^{\infty}\int_{\mathbb{S}^{N-1}}|\nabla u|^pr^{N-\alpha-1}d\vartheta dr\\&=C_3\int_{\mathbb{R}^N}\frac{|\nabla u|^p}{|x|^\alpha} dx.
\end{align*}
Hence, the weighted Hardy inequality \eqref{eq-p02} follows from \eqref{eqn-t1} using the above estimate.
\qed

\begin{cor} Let $2 \le k \le N$ and $x=(y,z)\in \mathbb{R}^k\times\mathbb{R}^{N-k}$. Let $p\in (1,\infty)$ and $\alpha\in\mathbb{R}$ be such that $k>p+\alpha$. 
Assume that $g \in L^{q}(\mathbb{S}^{N-1})$, where
\[
q = 
\begin{cases}
\dfrac{k-1}{p}, & \text{if } p < k-1, \\
1, & \text{if } p > k-1 ,
\end{cases}
\]
and  $q>1$ when $ p = k-1.$
Then by Theorem~\ref{thm-p2},
\begin{align*}
   \int_{\mathbb{R}^N}\frac{ g (y/|y|)}{|y|^{p+\alpha}}|u(x)|^p dx&= \int_{\mathbb{R}^{N-k}}dz\int_{\mathbb{R}^{k}}\frac{ g (y/|y|)}{|y|^{p+\alpha}}|u(y,z)|^p dy \\&\leq C\| g \|_{L^q(\mathbb{S}^{k-1})}\int_{\mathbb{R}^N}\frac{|\nabla u(x)|^p}{|y|^\alpha} dx, \quad\forall\,u\in \mathcal{C}_c^\infty(\mathbb{R}^N). 
\end{align*}
Thus, we obtain a Hardy-type inequality with cylindrical weights. For $g\equiv 1$, the above inequality reduces to the cylindrical Hardy inequality proved in \cite{Tarentello2002} for $\alpha=0$, and extended in \cite{cylin} to any $\alpha<k-p.$
\end{cor}

In order to prove Theorem \ref{sharp-wthm}, we use the Gagliardo-Nirenberg inequalities on the $(N-1)$-dimensional unit sphere, obtained by Dolbeault, Esteban, and Laptev \cite{Ari2014}.

\begin{lem}\cite[Lemma 5]{Ari2014}\label{G-N-ineq}
 Let $r=\frac{2(N-1)}{N-3}$ if $N>3,$ and $r=\infty$ if $N=2$ or $N=3.$ If $t\in (2,r)$, then there exists a concave increasing function $\mu:\mathbb{R}_+\rightarrow \mathbb{R}_+$ with
    $$\mu (\beta)=\beta,\quad \forall\, \beta\in \left[0,\frac{N-1}{t-2}\right],$$
    such that 
    \begin{equation}\label{gn}
       \|\nabla_\vartheta u\|^2_{L^2(\mathbb{S}^{N-1})} +\beta  \| u\|^2_{L^2(\mathbb{S}^{N-1})}\ge \mu(\beta)\,|\mathbb{S}^{N-1}|^\frac{t-2}{t}\| u\|^2_{L^t(\mathbb{S}^{N-1})},\quad \forall\, u\in H_1^2(\mathbb{S}^{N-1}). 
    \end{equation}
   Moreover, if $N>3$ and $t=\frac{2(N-1)}{N-3}$, then the above inequality also holds for any $\beta>0$ with $$\mu(\beta)=\min\left\{\beta,\frac{N-1}{t-2}\right\}.$$
\end{lem}
\smallskip

\noindent\textbf{Proof of Theorem \ref{sharp-wthm}:} Let $N\ge 2$ and $\alpha\in \mathbb{R}$ be such that $N> 2+\alpha$. Let $u\,\in \mathcal{C}_c^\infty(\mathbb{R}^N)$ and $x=(r,\vartheta)$ denote the polar coordinates in $\mathbb{R}^N.$  Let $t>2$ be as in Lemma \ref{G-N-ineq}, and define $q=\frac{t}{t-2}.$
By H\"older's inequality for the conjugate pair $(\frac{t}{t-2},\frac{t}{2})$, we obtain
$$\int_{\mathbb{S}^{N-1}}  g (\vartheta)|u|^2d\vartheta \leq \| g \|_{L^q(\mathbb{S}^{N-1})} \|u\|^2_{L^{t}(\mathbb{S}^{N-1})}.$$
 Now, use \eqref{gn} to obtain
\[
\int_{\mathbb{S}^{N-1}}  g (\vartheta)|u|^2d\vartheta \leq \frac{\|g \|_{L^q(\mathbb{S}^{N-1})}  }{\mu(\beta)|\mathbb{S}^{N-1}|^{1/q}}\int_{\mathbb{S}^{N-1}}\left( \beta |u|^2+ |\nabla_\vartheta u|^2\right)d\vartheta,\quad \beta>0.
\]
Multiply the above inequality by $r^{N-\alpha-3}$ and integrate over $(0,\infty)$ to obtain
\begin{align}
    \int_{\mathbb{R}^{N}} \frac{g(x/|x|)}{|x|^{2+\alpha}} |u|^pdx \leq  \frac{\|g \|_{L^q(\mathbb{S}^{N-1})}  }{\mu(\beta)|\mathbb{S}^{N-1}|^{1/q}}\int_0^{\infty}\int_{\mathbb{S}^{N-1}}\left( \beta |u|^2+ |\nabla_\vartheta u|^2\right)r^{N-\alpha-3} dr d\vartheta\label{gn2}
\end{align}
Moreover, apply Lemma \ref{weighted-Hardy} to obtain
\begin{equation}\label{one-dimes}
     \int_0^{\infty}|u|^2r^{N-\alpha-3} dr\le \frac{4}{(N-\alpha-2)^2}\int_0^{\infty}|\partial_ru|^2r^{N-\alpha-1}dr.
\end{equation}
Consequently,
\begin{equation}\label{gn-1}
      \int_{\mathbb{R}^{N}} \frac{g(x/|x|)}{|x|^{2+\alpha}} |u|^pdx \leq   \frac{\|g \|_{L^q(\mathbb{S}^{N-1})}  }{\mu(\beta)|\mathbb{S}^{N-1}|^{1/q}}\int_0^{\infty}\int_{\mathbb{S}^{N-1}}\left( \frac{4\beta}{(N-\alpha-2)^2} |\partial_r u|^2+ \frac{|\nabla_\vartheta u|^2}{r^2}\right)r^{N-\alpha-1} dr d\vartheta.
\end{equation}

 \noindent $(1)$ 
Let $2N\alpha<(N-\alpha-2)^2$  and $g\in L^q(\mathbb{S}^{N-1})$, where $q=\frac{(N-\alpha-2)^2}{2(N-1)}+1$. Then $$t=\frac{2q}{q-1}=2+\frac{4(N-1)}{(N-\alpha-2)^2}>2.$$ Since $2N\alpha<(N-\alpha-2)^2,$ we also have $$t <\frac{2(N-1)}{N-3}.$$  
Thus, by Lemma \ref{G-N-ineq}, inequality \eqref{gn-1} holds with
 $$\mu (\beta)=\beta,\quad \forall\, \beta\in \left[0,\;\frac{N-1}{t-2}=\frac{(N-\alpha-2)^2}{4}\right].$$
Now, by choosing $\beta=\frac{(N-\alpha-2)^2}{4}$ in \eqref{gn-1} and using the fact that  $\mu(\beta)=\beta$, we obtain
$$ \int_{\mathbb{R}^{N}} \frac{g(x/|x|)}{|x|^{2+\alpha}} |u|^pdx\le  \frac{4\|g \|_{L^q(\mathbb{S}^{N-1})}}{|\mathbb{S}^{N-1}|^{1/q} \,(N-\alpha-2)^2}\int_0^{\infty}\int_{\mathbb{S}^{N-1}}\left( |\partial_r u|^2+ \frac{|\nabla_\vartheta u|^2}{r^2}\right)r^{N-\alpha-1} dr d\vartheta.$$
This completes the proof of the inequality \eqref{shrp-whardy}.

\smallskip

\noindent $(2)$ Let $2N\alpha<(N-\alpha-2)^2,\,N>3$,  and $g\in L^q(\mathbb{S}^{N-1})$, where $q=\frac{N-1}{2}$. Then $$t=\frac{2q}{q-1}=\frac{2(N-1)}{N-3}>2.$$ 
Therefore, by Lemma \ref{G-N-ineq}, inequality \eqref{gn2} holds with 
\begin{equation*}
    \mu(\beta)=\min\left\{\beta,\frac{(N-1)(N-3)}{4}\right\}.
\end{equation*}
 Choose $\gamma_0>0$ such that $$\gamma_0=\frac{(N-\alpha-2)^2}{(N-1)(N-3)}.$$ Since $2N\alpha<(N-\alpha-2)^2,$ we have $$\gamma_0>1.$$
 Now, from \eqref{gn2}, we have 
\begin{align}
    \int_{\mathbb{R}^{N}} \frac{g(x/|x|)}{|x|^{2+\alpha}} |u|^2dx &\leq  \frac{\|g \|_{L^q(\mathbb{S}^{N-1})}  }{\mu(\beta)|\mathbb{S}^{N-1}|^{1/q}}\int_0^{\infty}\int_{\mathbb{S}^{N-1}}\left( \beta\gamma_0 |u|^2-\beta(\gamma_0-1) |u|^2+\nabla_\vartheta u|^2\right)r^{N-\alpha-3} dr d\vartheta\nonumber\\ &\leq   \frac{\|g \|_{L^q(\mathbb{S}^{N-1})}  }{\mu(\beta)|\mathbb{S}^{N-1}|^{1/q}}\int_0^{\infty}\int_{\mathbb{S}^{N-1}}\left( \frac{4\beta\gamma_0 }{(N-\alpha-2)^2} |\partial_r u|^2+ \frac{|\nabla_\vartheta u|^2}{r^2}\right)r^{N-\alpha-1} dr d\vartheta \nonumber\\&- \beta(\gamma_0-1) \frac{\|g \|_{L^q(\mathbb{S}^{N-1})}  }{\mu(\beta)|\mathbb{S}^{N-1}|^{1/q}}\int_0^{\infty}\int_{\mathbb{S}^{N-1}}  |u|^2r^{N-\alpha-3} dr d\vartheta\nonumber
\end{align}
where the last inequality follows from \eqref{one-dimes}. Thus, by choosing $$\beta=\frac{(N-\alpha-2)^2}{4\gamma_0}=\frac{(N-1)(N-3)}{4}$$ and using $\mu(\beta)=\beta,$ we obtain
\begin{multline*}
    \int_{\mathbb{R}^{N}} \frac{g(x/|x|)}{|x|^{2+\alpha}} |u|^2dx \leq   \frac{\|g \|_{L^q(\mathbb{S}^{N-1})}  }{\beta|\mathbb{S}^{N-1}|^{1/q}}\int_0^{\infty}\int_{\mathbb{S}^{N-1}}\left( |\partial_r u|^2+ \frac{|\nabla_\vartheta u|^2}{r^2}\right)r^{N-\alpha-1} dr d\vartheta \\ -(\gamma_0-1) \frac{\|g \|_{L^q(\mathbb{S}^{N-1})}  }{|\mathbb{S}^{N-1}|^{1/q}}\int_0^{\infty}\int_{\mathbb{S}^{N-1}}  |u|^2r^{N-\alpha-3} dr d\vartheta.
\end{multline*}
Hence, the inequality \eqref{sharp-whardy2} follows.

\smallskip

\noindent $(3)$ Let $2N\alpha\ge (N-\alpha-2)^2,\,N>3$, and $g\in L^q(\mathbb{S}^{N-1})$, where $q=\frac{N-1}{2}$. Then $$q=\frac{2(N-1)}{N-3}>2.$$ Thus, by Lemma \ref{G-N-ineq}, inequality \eqref{gn-1} holds with $$\mu(\beta)=\min\left\{\beta,\frac{(N-1)(N-3)}{4}\right\}.$$ 
Since $2N\alpha\ge (N-\alpha-2)^2,$ we also have $$\frac{(N-\alpha-2)^2}{4}\le \frac{(N-1)(N-3)}{4}.$$  
Now, by choosing $\beta=\frac{(N-\alpha-2)^2}{4}$ in \eqref{gn-1} and using $\mu(\beta)=\beta$ , we obtain
$$ \int_{\mathbb{R}^{N}} \frac{g(x/|x|)}{|x|^{2+\alpha}} |u|^pdx\le  \frac{4\|g \|_{L^q(\mathbb{S}^{N-1})}}{|\mathbb{S}^{N-1}|^{1/q} \,(N-\alpha-2)^2}\int_0^{\infty}\int_{\mathbb{S}^{N-1}}\left( |\partial_r u|^2+ \frac{|\nabla_\vartheta u|^2}{r^2}\right)r^{N-\alpha-1} dr d\vartheta.$$
This completes the proof.
\qed

\smallskip

Next, we consider the weighted Hardy inequality \eqref{hom-we} for a different class of function spaces for $g$, which may differ from those in Theorem \ref{thm-p2} and Theorem \ref{shrp-whardy}. Our proof follows the same lines as the proof of Theorem 1.7 in \cite{ari2015}.

\begin{thm}\label{theorem2.1}
    Let $N\ge 1,\,p\in (1,\infty),$ and $\alpha\in \mathbb{R}$ be such that $N> p+\alpha>0$. If $0\le g\in L^\frac{N}{p+\alpha}(\mathbb{S}^{N-1})$, 
then the following weighted Hardy inequality holds:
\begin{equation}\label{eq1-thm2.1}
    \int_{\mathbb{R}^N}\frac{ g (x/|x|)}{|x|^{p+\alpha}}|u(x)|^p dx \leq C \int_{\mathbb{R}^N}\frac{|\nabla u(x)|^p}{|x|^\alpha} dx, \quad\forall\,u\in \mathcal{C}_c^\infty(\mathbb{R}^N),
\end{equation}
where  $$C=\left(\frac{p}{N-p-\alpha}\right)^p\frac{\|g\|_{L^\frac{N}{p+\alpha}(\mathbb{S}^{N-1})}}{|\mathbb{S}^{N-1}|^\frac{p+\alpha}{N}}.$$
The constant in \eqref{eq1-thm2.1} is sharp, in the sense that it is attained for $g\equiv 1.$
\end{thm}
\begin{proof}
     Let $N>p+\alpha>0$ and $0\le g \in L^\frac{N}{p+\alpha}(\mathbb{S}^{N-1}).$ Let $x=(r,\vartheta)\in \mathbb{R}^N$ denote the polar coordinates in $\mathbb{R}^N$ and $u\in \mathcal{C}_c^\infty(\mathbb{R}^N)$. Applying the Hardy-Littlewood inequality (see $(1)$ of Proposition \ref{prop1}) and the fact that $(|u|^p)^*=|u^*|^p$ (see \cite{Edmunds2004}), we get
 \begin{equation}\label{eqn1-fract-p}
\int_{\mathbb{R}^N}\frac{ g (x/|x|)}{|x|^{p+\alpha}}|u(x)|^p dx\leq \int_{\mathbb{R}^N}\left(\frac{ g (x/|x|)}{|x|^{p+\alpha}}\right)^*|u^*(x)|^p dx.     
 \end{equation}
Next, we compute the symmetric decreasing rearrangement $\left(\frac{ g (x/|x|)}{|x|^{p+\alpha}}\right)^*.$
Observe that
\begin{align*}
   \left |\left\{y\in \mathbb{R}^N:\frac{ g (y/|y|)}{|y|^{p+\alpha}}>t\right\}\right|&= \left |\left\{y\in \mathbb{R}^N:|y|<\left(\frac{ g (y/|y|)}{t}\right)^\frac{1}{p+\alpha}\right\}\right|\\&=\int_{\mathbb{S}^{N-1}}\int_0^{( g(\vartheta) /t)^\frac{1}{p+\alpha}}r^{N-1}drd\vartheta\\&=\frac{t^\frac{-N}{p+\alpha}}{N}\int_{\mathbb{S}^{N-1}} g ^\frac{N}{p+\alpha}(\vartheta)d\vartheta,\quad \forall \,t\in (0,\infty).
\end{align*}
Now, using the definition of the symmetric decreasing rearrangement of a set (see \eqref{symm-set}), we obtain
$$\left\{y\in\mathbb{R}^N:\frac{| g (y/|y|)|}{|y|^{p+\alpha}}>t\right\}^*=\left\{y\in \mathbb{R}^N:|\mathbb{S}^{N-1}||y|^N<t^\frac{-N}{p+\alpha}\int_{\mathbb{S}^{N-1}} g ^\frac{N}{p+\alpha}(\vartheta)d\vartheta \right\}.$$
Thus, we conclude that
\begin{align*}
    \left(\frac{ g (x/|x|)}{|x|^{p+\alpha}}\right)^*&=\int_0^\infty\chi_{\left\{y\in\mathbb{R}^N:\frac{| g (y/|y|)|}{|y|^{p+\alpha}}>t\right\}^*}(x)\,dt\\&=\frac{1}{|\mathbb{S}^{N-1}|^{(p+\alpha)/N}|x|^{p+\alpha}}\left(\int_{\mathbb{S}^{N-1}} g ^\frac{N}{p+\alpha}(\vartheta)d\vartheta\right)^\frac{p+\alpha}{N}.
    \end{align*}
Consequently, from \eqref{eqn1-fract-p}, we obtain
$$\int_{\mathbb{R}^N}\frac{ g (x/|x|)}{|x|^{p+\alpha}}|u(x)|^p dx\leq \frac{\| g \|_{L^\frac{N}{p+\alpha}(\mathbb{S}^{N-1})}}{|\mathbb{S}^{N-1}|^{(p+\alpha)/N}}\int_{\mathbb{R}^N}\frac{|u^*(x)|^p}{|x|^{p+\alpha}}dx.$$
Since $N>p+\alpha$, it follows from \cite[Corollary 1.2.9]{Hardybook2015} that $$\int_{\mathbb{R}^N}\frac{|u^*(x)|^p}{|x|^{p+\alpha}}dx\leq \left(\frac{p}{N-p-\alpha}\right)^p\int_{\mathbb{R}^N}\frac{|\nabla u^*(x)|}{|x|^\alpha}dx.$$ Therefore,
$$\int_{\mathbb{R}^N}\frac{ g (x/|x|)}{|x|^{sp}}|u(x)|^p dx\leq \left(\frac{p}{N-p-\alpha}\right)^p\frac{\| g \|_{L^\frac{N}{p+\alpha}(\mathbb{S}^{N-1})}}{|\mathbb{S}^{N-1}|^{(p+\alpha)/N}}  \int_{\mathbb{R}^N}\frac{|\nabla u^*(x)|}{|x|^\alpha}dx.$$
Now, \eqref{eq1-thm2.1} follows from the weighted P\'olya and Szeg\"o inequality (see $(2)$ of Proposition \ref{prop1}). 
\end{proof}
\begin{rem} $(i)$ For $p=2$ with $\alpha=0$ Theorem \ref{theorem2.1} is established in \cite[Theorem 1.7]{ari2015}, and hence it generalises that result to all $p\in (1,\infty)$ and $\alpha\in \mathbb{R}.$

\smallskip
$(ii) $ Let $p\in (1,N-1)$ and suppose that $N>p+\alpha>0$. If $\alpha< \frac{p}{N-1},$ then Theorems \ref{thm-p2} and \ref{sharp-wthm} are stronger than Theorem \ref{theorem2.1}, since in this case we have the strict embedding
 $$L^\frac{N}{p+\alpha}(\mathbb{S}^{N-1})\subsetneq L^\frac{N-1}{p}(\mathbb{S}^{N-1}).$$ On the other hand, when $\alpha>\frac{p}{N-1},$ Theorem \ref{theorem2.1} becomes more general, as it admits a wider class of functions than $L^\frac{N-1}{p}(\mathbb{S}^{N-1}).$
\end{rem}

\smallskip

Next, we prove Theorem \ref{frac-hardy}, which establishes the weighted fractional Hardy inequality \eqref{eqn-frac-p}. The underlying idea is the same as in the proof of Theorem \ref{theorem2.1}.

\noindent \textbf{Proof of Theorem \ref{frac-hardy}:}
 Let $N>sp$ and $ 0\le g \in L^\frac{N}{sp}(\mathbb{S}^{N-1}).$ Let $x=(r,\vartheta)\in \mathbb{R}^N$ denote the polar coordinates in $\mathbb{R}^N$ and $u\in \mathcal{C}_c^\infty(\mathbb{R}^N)$. Applying the Hardy-Littlewood inequality (see $(1)$ of Proposition \ref{prop1}), we get
 \begin{equation}\label{eqn1-fract}
\int_{\mathbb{R}^N}\frac{ g (x/|x|)}{|x|^{sp}}|u(x)|^p dx\leq \int_{\mathbb{R}^N}\left(\frac{ g (x/|x|)}{|x|^{sp}}\right)^*|u^*(x)|^p dx.     
 \end{equation}
 We now proceed to compute the symmetric decreasing rearrangement $\left(\frac{ g (x/|x|)}{|x|^{sp}}\right)^*.$
Observe that
\begin{align*}
   \left |\left\{y\in \mathbb{R}^N:\frac{ g (y/|y|)}{|y|^{sp}}>t\right\}\right|&= \left |\left\{y\in \mathbb{R}^N:|y|<\left[\frac{ g (y/|y|)}{t}\right]^{1/sp}\right\}\right|\\&=\int_{\mathbb{S}^{N-1}}\int_0^{( g(\vartheta) /t)^{1/sp}}r^{N-1}drd\vartheta\\&=\frac{1}{N}t^{-N/sp}\int_{\mathbb{S}^{N-1}} g ^{N/sp}(\vartheta)d\vartheta,\quad \forall \,t\in (0,\infty).
\end{align*}
Now, by the definition of the symmetric decreasing rearrangement of a set, we obtain
$$\left\{y\in\mathbb{R}^N:\frac{| g (y/|y|)|}{|y|^{sp}}>t\right\}^*=\left\{y\in \mathbb{R}^N:|\mathbb{S}^{N-1}||y|^N<t^{-N/sp}\int_{\mathbb{S}^{N-1}} g ^{N/sp}(\vartheta)d\vartheta \right\}.$$
Therefore, 
\begin{align*}
    \left(\frac{ g (x/|x|)}{|x|^{sp}}\right)^*&=\int_0^\infty\chi_{\left\{y\in\mathbb{R}^N:\frac{| g (y/|y|)|}{|y|^{sp}}>t\right\}^*}(x)\,dt\\&=\int_0^\infty\chi_{\left\{y\in\mathbb{R}^N:|\mathbb{S}^{N-1}||y|^N<t^{-N/sp}\int_{\mathbb{S}^{N-1}} g ^{N/sp}(\vartheta)d\vartheta \right\}}(x)\,dt\\&=\frac{1}{|\mathbb{S}^{N-1}|^{sp/N}|x|^{sp}}\left(\int_{\mathbb{S}^{N-1}} g ^{N/sp}(\vartheta)d\vartheta\right)^{sp/N}.
    \end{align*}
Thus, from \eqref{eqn1-fract}, we obtain
$$\int_{\mathbb{R}^N}\frac{ g (x/|x|)}{|x|^{sp}}|u(x)|^p dx\leq \frac{\| g \|_{L^\frac{N}{sp}(\mathbb{S}^{N-1})}}{|\mathbb{S}^{N-1}|^{sp/N}}\int_{\mathbb{R}^N}\frac{|u^*(x)|^p}{|x|^{sp}}dx.$$
Since $N>sp$, it follows from \eqref{frac-norm} that $$\int_{\mathbb{R}^N}\frac{|u^*(x)|^p}{|x|^{sp}}dx\leq \Lambda_{N,s,p}\int_{\mathbb{R}^N}\int_{\mathbb{R}^N}\frac{|u^*(x)-u^*(y)|^p}{|x-y|^{N+sp}}dxdy.$$ Consequently,
$$\int_{\mathbb{R}^N}\frac{ g (x/|x|)}{|x|^{sp}}|u(x)|^p dx\leq \Lambda_{N,s,p}\frac{\| g \|_{L^\frac{N}{sp}(\mathbb{S}^{N-1})}}{|\mathbb{S}^{N-1}|^{sp/N}}  \int_{\mathbb{R}^N}\int_{\mathbb{R}^N}\frac{|u^*(x)-u^*(y)|^p}{|x-y|^{N+sp}}dxdy.$$
Hence, inequality \eqref{eqn-frac-p} follows from the fractional P\'olya and Szeg\"o inequality (see $(3)$ of Proposition \ref{prop1}). This completes the proof.
\qed

\begin{center}
	{\bf Acknowledgments}
\end{center}
The author acknowledges the support received from the Core Research Grant (CRG/2023/005344). The author would like to thank Prof. T.V. Anoop (IIT Madras) and Dr. Ujjal Das (BCAM, Spain) for discussions on the subject.

\bibliographystyle{abbrvurl}
\bibliography{Reference}


\end{document}